\documentclass{amsart}
\usepackage{amssymb}
\usepackage{amsfonts}
\usepackage{eufrak}
\usepackage{eucal}
\usepackage{txfonts}
\usepackage{amsmath}
\usepackage{enumerate}
\usepackage{comment}
\usepackage{hyperref}

\newtheorem{thm}{Theorem}[section]
\newtheorem{cor}[thm]{Corollary}

\newtheorem{lem}[thm]{Lemma}

\theoremstyle{definition}

\theoremstyle{remark}
\newtheorem{rem}[thm]{Remark}
\numberwithin{equation}{section}

\begin{document}
	\title[]
	{pseudolocality theorems of Ricci flows on incomplete manifolds
	}
	
	\author{Liang Cheng}

	

	
	\thanks{Liang Cheng's  Research partially supported by
		Natural Science Foundation of China 12171180
	}
	
	\address{School of Mathematics and Statistics $\&$ Hubei Key Laboratory of Mathematical Sciences, Central  China Normal University, Wuhan, 430079, P.R.China}
	
	\email{chengliang@ccnu.edu.cn }

	\maketitle

	\begin{abstract}
		In this paper we study the pseudolocality theorems of Ricci flows on incomplete manifolds. We prove that if a relatively compact ball in an incomplete manifold has the small scalar curvature lower bound and almost Euclidean isoperimetric constant, or almost Euclidean local $\boldsymbol{\nu}$ constant, then we can construct a solution of Ricci flow in a smaller ball for which the pseudolocality theorems hold on a uniform time interval. We also give two applications. First, we prove the short-time existence of Ricci flows on complete manifolds
		with scalar curvature bounded below uniformly and almost Euclidean isoperimetric
		inequality holds locally. Second, we obtain a rigidity theorem that any complete manifold
		with nonnegative scalar curvature and Euclidean isoperimetric
		inequality must be isometric to the Euclidean space.
	\end{abstract}

	\section{Introduction}
	
	The Ricci flow is a geometric evolution equation introduced by Hamilton \cite{H1}, which deforms a Riemannian manifold by the Ricci curvature
	$$\frac{\partial}{\partial t}g(t)=-2Rc(g(t)).$$
	In [23], Perelman proved an interior curvature
	estimate for Ricci flows known as the pseudolocality theorem,
	which becomes an important tool in the study the Ricci flows and even many problems in Riemannian geometry.
	The celebrated Perelman's pseudolocality states that
	\begin{thm}[Perelman's pseudolocality theorem \cite{P1}]\label{perelman_pseudo_locality}
		For every $\alpha>0$ and $n \geq 2$ there exist $\delta>0$ and $\epsilon_0>0$ depending only on $\alpha$ and $n$ with the following property. Let $\left(M, g(t)\right), t \in\left[0,\left(\epsilon r_0\right)^2\right]$, where $\epsilon \in\left(0, \epsilon_0\right]$ and $r_0 \in(0, \infty)$, be a complete solution of the Ricci flow with bounded curvature   and let $x_0 \in M$ be a point such that
		\begin{equation}\label{small_scalar_cur}
			R(x, 0) \geq-r_0^{-2}
		\end{equation}
		for $x \in B_{g_0}\left(x_0, r_0\right)$
		and
		\begin{equation}\label{almost_euclidean_iso}
			\left(\text { Area }_{g_0}(\partial \Omega)\right)^n \geq(1-\delta) c_n\left(\operatorname{Vol}_{g_0}(\Omega)\right)^{n-1}
		\end{equation}
		for any regular domain $\Omega \subset B_{g_0}\left(x_0, r_0\right)$, where $c_n \doteqdot n^n \omega_n$ is the Euclidean isoperimetric constant. Then we have the interior curvature estimate
		$$
		|\operatorname{Rm}|(x, t) \leq \frac{\alpha}{t}+\frac{1}{\left(\epsilon r_0\right)^2}
		$$
		for $x \in \mathcal{M}$ such that $d_{g(t)}\left(x, x_0\right)<\epsilon r_0$ and $t \in\left(0,\left(\epsilon r_0\right)^2\right] $.
	\end{thm}
The original version of Perelman's pseudolocality theorem \cite{P1} was proved under the
assumption of the manifold being closed. In the complete and noncompact case, this result
was verified by Chau, Tam and Yu\cite{CTY}. Tian and Wang \cite{TW} proved another version of
	pseudolocality theorem in which they showed that the conditions (\ref{small_scalar_cur}) and (\ref{almost_euclidean_iso}) in Theorem \ref{perelman_pseudo_locality} can be replaced
	by small Ricci curvature and almost Euclidean volume ratio.
	 Subsequently, Wang \cite{w2} improved both Perelman and Tian-Wang's pseudolocality theorems and proved that if for each $\alpha>0$, there exists $\delta=\delta
	(\alpha,n)$ such that
	\begin{equation}\label{wang}
	\boldsymbol{\nu}\left(B_{g_0}\left(x_0, \delta^{-1} \sqrt{T}\right), g_0, T\right) \geq-\delta^2
\end{equation}
	for the complete  Ricci flow $\left(M, g(t)\right)|_{ 0 \leq t \leq T}$ with bounded curvature  , then
	$$|R m|(x, t) \leq \frac{\alpha}{t}$$
	for  $(x,t) \in B_{g(t)}\left(x_0, \alpha^{-1} \sqrt{t}\right)\times (0, T]$,
	where $\boldsymbol{\nu}$  is the localized Perelman's entropy  which is defined as
	\begin{equation}\label{local_mu}
		\begin{aligned}
			\boldsymbol{\mu}(\Omega, g, \tau):&=\inf\limits_{\varphi \in \mathcal{S}(\Omega)} \mathcal{W}(\Omega, g, \varphi, \tau)\\
			&=\inf\limits_{\varphi \in \mathcal{S}(\Omega)} \left\{-n-\frac{n}{2} \log (4 \pi \tau)+\int_{\Omega}\left\{\tau\left(R \varphi^2+4|\nabla \varphi|^2\right)-2 \varphi^2 \log \varphi \right\} d vol_{g}\right\},\\
			\boldsymbol{\nu}(\Omega, g, \tau):&=\inf\limits_{s\in(0,\tau]}\boldsymbol{\mu}(\Omega, g, s),
		\end{aligned}
	\end{equation}
	and
	$\mathcal{S}(\Omega):=\left\{\varphi \mid \varphi \in W_0^{1,2}(\Omega),  \varphi \geq 0,  \int_{\Omega} \varphi^2 d v=1\right\}$.
	One may also see \cite{CMZ} for another proof of above pseudolocality theorems based on Bamler's $\epsilon$-regularity theorem \cite{RB1}. In \cite{RB3}  Bamler  also obtained a backward version pseudolocality theorem.
	
	Note that the above pseudolocality theorems are not really local results since they all
	require completeness and bounded curvature of the Ricci flows. So a natural question is that whether the pseudolocality theorems still hold without the completeness assumption?
	However, the following example due to Peter Topping indicates  that not
	all solutions of Ricci flow starting from an incomplete metric have the
	pseudolocality theorem; see Example $21.5$ and Theorem $21.6$ in \cite{RFV3} :
	Consider a cylinder
	$$
	\mathcal{S}^1(r) \times[-1,1]
	$$
	with the flat product metric, where $\mathcal{S}^1(r)$ denotes the circle of radius $r$. We cap each of the two ends of the cylinder with a disc $D^2$ and use a cutoff function to smoothly blend the cylinder metric with the round hemisphere $\mathcal{S}_{+}^2(r)$ in thin collars about their boundaries to construct a rotationally symmetric surface $(\Sigma^2,g_0^r)$ with nonnegative curvature. Let $g^r(t)$ be the solution of Ricci flow on $\Sigma$ with initial data $g_0^r$.
	Now define incomplete solution of the Ricci flow as follows:
	Let
	$$
	\mathcal{M}^2 \doteqdot\left(-\frac{3}{5}, \frac{3}{5}\right) \times\left(-\frac{3}{5}, \frac{3}{5}\right)
	$$
	and define the (into) local covering map
	$$
	\phi: \mathcal{M}^2 \rightarrow \mathcal{S}^1(r) \times[-1,1] \subset \Sigma^2
	$$
	by
	$$
	\phi(x, y)=(\pi(x), y)
	$$
	where  $\pi: \mathbb{R} \rightarrow \mathcal{S}^1(r)$ denote the standard covering map given by $\pi(x)=[x]$ is the equivalence class of $x \bmod 2 \pi r$. By the Gauss-Bonnet theorem we have
	$$
	\frac{d}{d t} \operatorname{Area}_{g^r(t)}(\Sigma)=-\int_{\Sigma} R_{g^r(t)} d \mu_{g^r(t)}=-8 \pi,
	$$
	so that
	$$
	\operatorname{Area}_{g^r(t)}(\Sigma)=\operatorname{Area}_{g_0^r}(\Sigma)-8 \pi t.
	$$
In fact, Hamilton \cite{H3}
proved that metrics on $\mathcal{S}^2$ with nonnegative curvature shrink to round points under the Ricci flow, we have  $\lim\limits _{t \to T^r}\left(\inf _{x \in \mathcal{M}} R_{g^r}(x, t)\right)=\infty$ where $T^r=\frac{1}{8 \pi} \operatorname{Vol}_{g_0^r}(\Sigma)$.	
	When $r$ is sufficient small,
	the pseudolocality theorems clearly do not hold for the incomplete Ricci flow $\left(\mathcal{M}^2, g_{\mathcal{M}}^r(t)\right)$ even its initial metric is flat, where $g_{\mathcal{M}}^r(t) \doteqdot \phi^* g^r(t)$.
	However,
	the Ricci flow starting from an incomplete initial metric always may not just have one solution. So Topping's example does not imply we could not
	always find just one solution of Ricci flow starting from an incomplete
	metric for which the pseudolocality holds.
	Actually, for Topping's example, pseudolocality theorems obviously hold for the flat solution on $\mathcal{M}^2$ with inital metric $g_{\mathcal{M}}^r(0)$.
	
 At first sight, one can easily perform a conformal change to a relatively compact ball so that the resulting new metric is complete and has the bounded sectional curvature; see Theorem \ref{Ho}.  By using Shi's local existence theorem for the Ricci flow of noncompact manifold, we have a solution of Ricci flow which exists on time interval $[0,T]$ ; see \cite{Sh}.
  Then one can restrict the flow to a smaller ball unchanged to has a  Ricci flow and pseudolocality holds on interval $[0,T]$.    However, owing to Shi's local existence theorem,  $T$ is dependent on the bound of sectional curvature the for the local metric of initial time.  
  This version of pseudolocality is not our purpose since it is not sufficient to get a solution of Ricci flow for noncompact manifolds which may not have bounded curvature; see Theorem \ref{existence} and Theorem \ref{rigidity} below. 

		In this paper,  we use an inductive conformal changing method, which was introduced in \cite{T3},  to show that
 if a relatively compact ball
	contained in an incomplete manifold satisfying (\ref{wang}),  or (\ref{small_scalar_cur}) and (\ref{almost_euclidean_iso}), then we can construct a solution of Ricci flow in a smaller ball for which  the pseudolocality theorems hold on a uniform time interval $[0,T]$ with
	$T$  depending only on $\alpha$ and the dimension.
	
	\begin{thm}\label{main_result}
		For each $\alpha>0$ and $n\ge 2$, there exist $\delta=\delta
		(\alpha,n)$ and $\epsilon(\alpha,n)$ with the following properties.
		Suppose $\left(M, g_0\right)$ is a smooth $n$-dimensional Riemannian manifold (not necessarily complete) such that $B_{g_0}\left(x_0, \delta^{-1} \sqrt{T}\right)\Subset M$ and
		$$
		\boldsymbol{\nu}\left(B_{g_0}\left(x_0, \delta^{-1} \sqrt{T}\right), g_0, T\right) \geq-\delta^2
		$$
		for some $T>0$.
		Then for each $\eta \in(0, 1)$ there exists a smooth Ricci flow $g(t)$ on $B_{g_0}\left(x_0,  (1-\eta)\delta^{-1} \sqrt{T}\right)\times [0, (\epsilon\eta)^2T]$ with $g(0)=g_0$ satisfying
		\begin{equation}
			|R m|(x, t) \leq \frac{\alpha}{t}
		\end{equation}
		and
		\begin{equation}
			\inf _{\rho \in\left(0, \alpha^{-1} \sqrt{t}\right)} \frac{\operatorname{Vol}\left(B_{g(t)}(x, \rho)\right)}{\rho^{n}} \geq(1-\alpha) \omega_n
		\end{equation}\label{non_collapsing}
		for  $(x,t)\in B_{g_0}\left(x_0,  (1-\eta)\delta^{-1} \sqrt{T}\right)\times [0, (\epsilon\eta)^2T]$.
	\end{thm}
	
	As a corollary to theorem \ref{main_result}, we have the following pseudolocality theorem related to Perelman's version.
	
	\begin{thm}\label{main_result2}
		For each $\alpha>0$ and $n\ge 2$,  there exist $\delta=\delta
		(\alpha,n)$ and $\epsilon(\alpha,n)$ with the following properties.
		Suppose $\left(M, g_0\right)$ is a smooth $n$-dimensional Riemannian manifold (not necessarily complete) such that $B_{g_0}\left(x_0, r_0^2\right)\Subset M$ $$R(x) \geq-r_0^{-2} $$ for $x \in B_{g_0}\left(x_0, r_0\right)$
		and
		$$
		\left(\operatorname { Area }_{g_0}(\partial \Omega)\right)^n \geq(1-\delta)  n^n \omega_n\left(\operatorname{Vol}_{g_0}(\Omega)\right)^{n-1}
		$$
		for any regular domain $\Omega \subset B_{g_0}\left(x_0, r_0\right)$.
		Then for each $\eta \in(0, 1)$ there exists a smooth Ricci flow $g(t)$ on $B_{g_0}\left(x_0,(1-\eta) r_0\right) \times\left[0,\left(\epsilon \eta r_0\right)^2\right]$ with $g(0)=g_0$ satisfying
		$$
		|\mathrm{Rm}|(x,t) \leq \frac{\alpha}{t}+\frac{1}{\left(\epsilon \eta r_0\right)^2},
		$$
		and
		$$\inf _{\rho \in\left(0, \alpha^{-1} \sqrt{t}\right)} \frac{\operatorname{Vol}\left(B_{g(t)}(x, \rho)\right)}{\rho^{n}} \geq(1-\alpha) \omega_n$$
		for  $(x,t) \in B_{g_0}\left(x_0,(1-\eta) r_0\right) \times\left(0,\left(\epsilon \eta r_0\right)^2\right]$.
	\end{thm}
	
	The existence of solutions to Ricci flows  on noncompact manifolds with bounded sectional curvature was obtained by Shi \cite{Sh}. However, without imposing any conditions, the existence of the Ricci flows on general
	complete manifolds is expected to be not true. So it is interesting to find the solutions to Ricci flows exist on noncompact manifolds with unbounded curvature under
	some other reasonable conditions; one may see
	Cabezas-Rivas and Wilking \cite{CW}, Chau, Li and Tam \cite{CLT}, Lee and Topping \cite{LP}, Giesen
	and Topping \cite{GT1} \cite{GT2}, Hochard \cite{Ho}, Simon \cite{S1}\cite{S2}, Topping \cite{T1} and the references
	therein for more information.
	As the first application to our pseudolocality theorems for the incomplete case, we can apply they to prove the short-time existence of Ricci flow solutions, with possibly
	unbounded curvature at the initial time.

	\begin{thm}\label{existence}
		For each $\alpha>0$, $n\ge 2$ and $r_0>0$, there exist $\delta=\delta
		(\alpha,n)$ and $\epsilon(\alpha,n)$ with the following properties.
		Suppose $\left(M^n, g_0\right)$ is a smooth complete $n$-dimensional Riemannian manifold  such that  $$R(x) \ge -k $$
		and
		$$
		\left(\operatorname { Area }_{g_0}(\partial \Omega)\right)^n \geq(1-\delta) n^n \omega_n\left(\operatorname{Vol}_{g_0}(\Omega)\right)^{n-1}
		$$
		for any regular domain $\Omega \subset B_{g_0}(x,r_0)$ and all $x\in M$.
		Then there exists a complete smooth Ricci flow $g(t)$ with  $g(0)=g_0$ on
		$M\times [0, (\epsilon r_0')^2 ]$
		satisfying
		$$|R m|(x, t) \leq \frac{\alpha}{t}$$
		and
		$$\inf _{\rho \in\left(0, \alpha^{-1} \sqrt{t}\right)} \frac{\operatorname{Vol}\left(B_{g(t)}(x, \rho)\right)}{\rho^{n}} \geq(1-\alpha) \omega_n$$
		for $(x,t)\in M\times [0, (\epsilon r_0')^2 ]$, where $r_0'=\min\{r_0,\frac{1}{\sqrt{k}}\}$ if $k>0$ and $r_0'=r_0$ if $k\le 0$.
	\end{thm}

Notice that we can get	from the classical volume comparison theorem that if a complete Riemmannian manifold satisfying $Rc(g)\ge 0$ and $\frac{\operatorname{Vol}_{g}(B(p,r))}{r^n}\ge \omega_n$ for any $r>0$, then
it must be isometric to the Euclidean space. As an analogue, we have the
following rigidity theorem with respect to nonnegative scalar curvature.

\begin{thm}\label{rigidity}
	 Suppose $\left(M^n, g\right)$ is a smooth complete $n$-dimensional Riemannian manifold  such that
	 \begin{equation}\label{1}
	 R(x) \ge 0
	 \end{equation}
	 for all $x \in M$
	and
	\begin{equation}\label{2}
\left(\operatorname { Area }_{g}(\partial \Omega)\right)^n\ge  n^n \omega_n\left(\operatorname{Vol}_{g}(\Omega)\right)^{n-1}	
\end{equation}
		for any regular domain $\Omega \subset M$.
	Then $M$ is isometric to the Euclidean space.
\end{thm}

	With the extra condition that $\left(M, g\right)$ has the bounded sectional curvature, Theorem \ref{rigidity} can be easily obtained by the monotonicity of Perelman's $\mathcal{W}$-functional. Recall the  Perelman's $\mathcal{W}$-functional is defined as
	\begin{align}\label{w_func}
			\mathcal{W}=\int_M [\tau(|\nabla f|^2+ R)+f-n]Hd\mu,
		\end{align}
and we let $H=(4\pi\tau)^{-\frac{n}{2}}e^{-f} $ is the heat kernel of $(\frac{\partial}{\partial \tau}-\Delta+R)H=0$ with $\tau=T-t$.
 If there exists a complete solution to Ricci flow $g(t)$ with bounded sectional curvature on some time interval $[0,T]$, for which can be obtained by Shi local existence theorem \cite{Sh} if $\left(M, g(0)\right)$ has the bounded sectional curvature, then $|\mathcal{W}|<\infty$, $\mathcal{W}\le 0$ and  $\mathcal{W}=0$ at some time if and only
 if $(M,g(t))$ is isometric to the Euclidean space for any $t\in [0,T]$; see \cite{P1}, \cite{RFV3} or \cite{CTY}. Moreover, we have
 $\mathcal{W}\ge 0$ at $t=0$ if $(M,g(0))$ satisfies
 (\ref{1}) and (\ref{2}) and hence  $\mathcal{W}= 0$ at $t=0$
	and $(M,g(0))$ is isometric to the Euclidean space.
	 We also mention that He \cite{He} proved  Theorem \ref{existence} and Theorem \ref{rigidity}  with an extra condition
	$\liminf\limits_{d(x)\to \infty}d(x)^{-2}Rc(x)\ge -C$.

The present paper is organized as follows. In section 2 we recall some results which we shall use in the next sections. In section 3 we give the proofs of Theorem \ref{main_result} and Theorem \ref{main_result2}. In section 4 we give the proofs of Theorem \ref{existence} and Theorem \ref{rigidity}.

	\section{Preliminaries}
	
	In this section we recall some results which we shall use in the next sections. The first of these is a result of
 Li-Yau-Hamilton-Perelman type Harnack inequality by Wang \cite{w1} and Qi.S.Zhang \cite{Z}.
 \begin{thm}[Theorem 4.2 in \cite{w1}, Step 2 in the proof of Theorem 6.3.2 of \cite{Z}]\label{Harnack}
 Suppose $\left(M, g(t))\right|_{0 \leq t \leq T}$ is a complete $n$-dimensional Ricci flow with bounded sectional curvature, $\Omega$ is a bounded domain of $M$ with smooth boundary. Fix $\tau_T>0$, let $\varphi_T$ be the minimizer function of $\boldsymbol{\mu}\left(\Omega, g(T), \tau_T\right)$ for some $\tau_T>0 $. Starting from $u_T=\varphi_T^2$ at time $t=T$, let $u$ solve the conjugate heat equation
$$
\square^* u=\left(-\partial_t-\Delta+R\right) u=0 .
$$
Define
$$
\begin{aligned}
	\tau &:=\tau_T+T-t, \\
	f &:=-\frac{n}{2} \log (4 \pi \tau)-\log u, \\
	v&:=\left\{\tau\left(2 \Delta f-|\nabla f|^2+R\right)+f-n-\mu\right\} u,
\end{aligned}
$$
where $\boldsymbol{\mu}=\boldsymbol{\mu}\left(\Omega, g(T), \tau_T\right)$. Then we have
$$
v \leq 0.
$$
\end{thm}

	Next we recall the the following estimate for local $\boldsymbol{\mu}-$functional by the
	isoperimetric constant and lower bound of scalar curvature.
	
	\begin{thm}[Lemma 3.5 of \cite{w1}] \label{isoperimetric}
		Suppose $\Omega$ is a bounded domain in a Riemannian manifold $\left(M, g\right)$ with its scalar curvature satisfying
		$$R\ge -\underline{\Lambda}\text{\quad on \quad} \Omega.$$
		Let $\tilde{\Omega}$ be a ball in $\left(\mathbb{R}^n, g_E\right)$ such that $Vol(\tilde{\Omega})=Vol(\Omega)$.  Define
		$$
		\lambda:=\frac{\mathbf{I}(\Omega)}{\mathbf{I}_n},
		$$
		where $\mathbf{I}(\Omega)=\inf\limits_{D\Subset \Omega}\frac{Area(\partial D)}{Vol(D)^{\frac{n-1}{n}}}$
		is the isoperimetric constant with respect $g$ and $I_n$ is the
		isoperimetric constant of $n$-dimensional Euclidean space.
		Then we have
		$$
		\boldsymbol{\mu}(\Omega, g, \tau) \geq \boldsymbol{\mu}\left(\tilde{\Omega}, g_E, \tau \lambda^2\right)+n \log \lambda-\underline{\Lambda} \tau.
		$$
	\end{thm}
	
	The following  result of Hochard  that allows us to conformally change  an incomplete Riemannian
	metric at its extremities in order to make it complete and without changing it in the interior.
	
	\begin{thm} [Corollaire IV.1.2 in \cite{Ho}]\label{Ho}
		There exists $\sigma(n)$ such that given a Riemannian manifold $\left(N^n, g\right)$ with $|\operatorname{Rm}(g)| \leq \rho^{-2}$ throughout for some $\rho>0$, there exists a complete Riemannian metric $h$ on $N$ such that
		
		(1) $h \equiv g$ on $N_\rho:=\left\{x \in N: B_g(x, \rho) \Subset N\right\}$, and
		
		(2) $|\operatorname{Rm}(h)| \leq \sigma \rho^{-2}$ throughout $N$.
	\end{thm}

	We also recall the following lemma, which is one of the local ball inclusion results based on the distance distortion estimates of Hamilton and Perelman.
	
	\begin{thm}[Lemma 8.3 of \cite{P1}, Section of in \cite{H2}, Corollary 3.3 of \cite{T2} ]\label{shrinking}
		There exists a constant $\gamma=\gamma(n) $ depending only on $n$ such that the following is true. Suppose $\left(N^n, g(t)\right)$ is a Ricci flow for $t \in[0, S]$ with $g(0)=g_0$ and $x_0 \in N$ with $B_{g_0}\left(x_0, r\right) \Subset N$ for some $r>0$, and $\operatorname{Rc}(g(t)) \leq \frac{a}{t}$ on $B_{g_0}\left(x_0, r\right)$ for each $t \in(0, S]$. Then
		$$d_{g_0}(x,x_0)\le d_{g(t)}(x,x_0)+\gamma\sqrt{at}$$ on $B_{g_0}\left(x_0, r\right)$
		and hence
		$$
		B_{g(t)}\left(x_0, r-\gamma \sqrt{a t}\right) \subset B_{g_0}\left(x_0, r\right).
		$$	
	\end{thm}

 We also need the following lemma, which is a slight generalization of Theorem
	5.4 by Wang \cite{w1}, allows us to
	estimate the local $\boldsymbol{\nu}$-functional values under the Ricci flow.
	
	\begin{lem}\label{key_lemma}
		Let $\left\{\left(M, g(t)\right), s_1 \leq t \leq s_2 \right\}$ be a complete Ricci flow solution with bounded sectional curvature satisfying
		\begin{equation}
			\quad t \cdot R c(x, t) \leq(n-1) A, \quad \forall x \in B_{g(t)}\left(x_0, \sqrt{t}\right).
		\end{equation}
		Then for any $0\le s_1<s_2\le 1$, $\tau_1>0$, $0\le B<\frac{1}{2}$ and $0\le D<8-20B$, we have
		\begin{equation}\label{key}
			\boldsymbol{\nu}\left(\Omega_{s_2}^{\prime\prime}, g(s_2), \tau_{1}+1-s_2\right)-\boldsymbol{\nu}\left(\Omega_{s_1}, g(s_1), \tau_1+1-s_1\right) \geq-\left\{\frac{\tau_{1}+1}{10 A^2B^2}+e^{-1}\right\} \cdot\left\{e^{\frac{s_2-s_1}{10 A^2B^2}}-1\right\},
		\end{equation}
		where  $\Omega_{s_2}^{\prime\prime}=B_{g(s_2)}\left(x_0, 10A(1-2B)-2 A \sqrt{s_2}-DA\right)$ and $\Omega_{s_1}=B_{g(s_1)} \left(x_0, 10A-2 A \sqrt{s_1}-DA \right)$.
	\end{lem}
	\begin{proof}
		We follow the idea of \cite{w1}.
	Let $\psi$ be a cut-off function such that $\psi \equiv 1$ on $(-\infty, 1-B), \psi \equiv 0$ on $(1, \infty)$ and $-\frac{10}{B} \leq \psi^{\prime} \leq 0$ everywhere. Moreover, $\psi$ satisfies
		$$
		\psi^{\prime \prime} \geq-\frac{10}{B^2} \psi, \quad\left(\psi^{\prime}\right)^2 \leq \frac{10}{B^2} \psi.
		$$
		To construct $\psi$ we can take 	
		$$
		\psi(y)= \begin{cases}1, &  y \le 1-B ; \\
			1-\frac{2}{B^2}(y-1+B)^2,	  & 1-B \le  y \le 1-\frac{B}{2} ;\\
			\frac{2}{B^2}(y-1)^2, & 1-\frac{B}{2} \le  y \le  1; \\
			0, & y\ge 1 .,
		\end{cases}
		$$	
		and smooth it slightly.	
		Setting $$h(x)=\psi\left(\frac{d_{g(s_1)}(x,x_0)+2 A \sqrt{s_1}+DA}{10 A}\right).$$  For each $t \in[0,1]$, we define
		$\Omega_t:=B_{g(t)}\left(x_0, 10 A-2 A \sqrt{t}-DA\right),  \Omega_t^{\prime}:=B_{g(t)}\left(x_0, 10A(1-B)-2 A \sqrt{t}-DA \right)$.
		It follows from the definition that
		$$
		h(x)= \begin{cases}1, & \forall x \in \Omega_{s_1}^{\prime} ; \\ 0, & \forall x \in M \backslash \Omega_{s_1} .\end{cases}
		$$
		Then we have
		$$
		|\nabla \sqrt{h}|_{g(s_1)}^2=\frac{|\nabla h|_{g(s_1)}^2}{4 h}=\frac{\left(\psi^{\prime}\right)^2}{400 A^2 \psi} \leq \frac{1}{40 A^2B^2} .
		$$
		Next we define
		$$H(x,t)=e^{-\frac{t-s_1}{10A^2B^2}}\psi\left(\frac{d_{g(t)}(x,x_0)+2 A \sqrt{t}+10AB+DA}{10 A}\right)$$
		and 	 $ \Omega_t^{\prime\prime}:=B_{g(t)}\left(x_0, 10A(1-2B)-2 A \sqrt{t}-DA \right)$. Hence
		$$
		H(x,t)= \begin{cases}e^{-\frac{t-s_1}{10A^2B^2}}, & \forall x \in \Omega_{t}^{\prime\prime} ; \\ 0, & \forall x \in M \backslash \Omega^{\prime}_{t} .\end{cases}
		$$
		We have
		\begin{align*}
			&\left(\frac{\partial}{\partial t}-\Delta\right)\psi\left(\frac{d_{g(t)}(x,x_0)+2 A \sqrt{t}+10AB+DA}{10 A}\right)\\
			=&\frac{1}{10 A}\left(\left(\frac{\partial}{\partial t}-\Delta\right) d_{g(t)}(x,x_0)+\frac{A}{\sqrt{t}}\right) \psi^{\prime}-\frac{1}{(10 A)^2} \psi^{\prime \prime} \leq \frac{\psi}{10 A^2B^2}
		\end{align*}
		where we use $\left(\frac{\partial}{\partial t}-\Delta\right) d_{g(t)}(x,x_0)+\frac{A}{\sqrt{t}}\ge 0$ (see Lemma 8.3 in \cite{P1} or Section 17 in \cite{H2}). Then we have $\left(\frac{\partial}{\partial t}-\Delta\right) H \leq 0$.
		
		Let $\varphi$ be a minimizer for $\boldsymbol{\mu}_{s_2}=\boldsymbol{\mu}\left(\Omega_{s_2}^{\prime\prime}, g(s_2), \tau_1'+1-s_2\right)$ for some number $\tau_1'\in(s_2-1,\tau_1]$. Starting from $u_{s_2}=\varphi^2$, we solve the equation $\left(-\frac{\partial}{\partial t}-\Delta+R\right) u=0$ on $[s_1,s_2]$. Thus, we have
		$$
		\frac{d}{d t} \int_M u H=\int_M\left\{u\left(\frac{\partial}{\partial t}-\Delta\right) H+H \left(\frac{\partial}{\partial t}+\Delta-R\right)u\right\} \leq 0 .
		$$
		Since $u_{s_2}=0$
		outside of $\Omega_{s_2}^{\prime\prime}$ and integrate the above inequality yields that
		$$
		\left.\int_M u H\right|_{t=s_1} \geq\left.\int_M u H\right|_{t=s_2}=\left.\int_{\Omega_{s_2}^{\prime\prime}} u H\right|_{t=s_2}=\left.e^{-\frac{s_2-s_1}{10 A^2B^2}} \int_{\Omega_{s_2}^{\prime\prime}} u\right|_{t=s_2}=e^{-\frac{s_2-s_1}{10 A^2B^2}} .
		$$
		It follows that
		$$
		1 \geq\left.\int_{\Omega_{s_1}^{\prime}} u\right|_{t=s_1} \geq\left.\int_M u H\right|_{t=s_1} \geq e^{-\frac{s_2-s_1}{10 A^2B^2}},
		$$
		where we use $H\le 1$ at $t=s_1$ and $H \equiv 0$ outside of $\Omega_{s_1}^{\prime}$. Then we conclude that
		$\left.\int_{\Omega_{s_1}^{\prime}}u\right|_{t=s_1}\geq e^{-\frac{s_2-s_1}{10 A^2B^2}}$.
		We define
		 $S:=\left.\int_{\Omega_{s_1}} u h\right|_{t=s_1}  \leq\left.\int_M u\right|_{t=s_1}=1$, $v=\left\{\tau\left(2 \Delta f-|\nabla f|^2+R\right)+f-n-\mu_{s_2}\right\} u$ as in Theorem \ref{Harnack} and $\tilde{u}=\frac{u h}{S}$. Then  $\int_M \tilde{u}|_{t=s_1}=1$ and $\tilde{u}$ is supported on $\Omega_{s_1}$ at $t=s_1$. Denote
		$\tilde{f} =-\log \tilde{u}-\frac{n}{2} \log \left(4 \pi \tau_{s_1}\right)=f-\log h+\log S		$
		with $\tau_{s_1}=\tau'_1+1-s_1$, $\boldsymbol{\mu}_{s_1}=\boldsymbol{\mu}\left(\Omega_{s_1}, g(s_1), \tau_1'+1-s_1\right)$ and
		$\boldsymbol{\mu}_{s_2}=\boldsymbol{\mu}\left(\Omega_{s_2}^{\prime\prime}, g(s_2), \tau_1'+1-s_2\right)$.	We obtain
\begin{equation}\label{key2}
\begin{aligned}
	\boldsymbol{\mu}_{s_1}&\le \int_{\Omega_{s_1}}\left\{\tau_{s_1}\left(R+2 \Delta \tilde{f}-|\nabla \tilde{f}|^2\right)+\tilde{f}-n\right\} \tilde{u}|_{t=s_1} \\
	&=\boldsymbol{\mu}_{s_2}+\left\{\log S+\frac{1}{S} \int_{\Omega_{s_1}} v h|_{t=s_1} \right\}+\frac{1}{S} \int_{\Omega_{s_1}}\left\{4 \tau_{s_1}|\nabla \sqrt{h}|^2-h \log h\right\} u|_{t=s_1}  \\
	&\le\boldsymbol{\mu}_{s_2}+\frac{1}{S} \int_{\Omega_{s_1}\backslash \Omega'_{s_1}}\left\{4 \tau_{s_1}|\nabla \sqrt{h}|^2-h \log h\right\} u|_{t=s_1}  \\
	&\le\boldsymbol{\mu}_{s_2}+\left\{\frac{\tau'_{1}+1}{10 A^2B^2}+e^{-1}\right\} \cdot\frac{\int_{\Omega_{s_1}\backslash \Omega'_{s_1}} u|_{t=s_1} }{\int_{ \Omega'_{s_1}} u|_{t=s_1} }  \\
	&\leq \boldsymbol{\mu}_{s_2}+ \left\{\frac{\tau'_{1}+1}{10 A^2B^2}+e^{-1}\right\} \cdot\left\{e^{\frac{s_2-s_1}{10 A^2B^2}}-1\right\},
\end{aligned}
\end{equation}
where we use $v\le 0$ by Theorem \ref{Harnack}, $h\equiv 1$ on $\Omega'_{s_1}$ and $S\le 1$ in the above inequalities. Then (\ref{key})
follows by taking infimum of $\tau_1'$ on $(s_2-1,\tau_1]$ in (\ref{key2}).

	\end{proof}

	\section{The proofs of pseudolocality theorems on incomplete manifolds}	
	
		Before present the proofs of Theorem \ref{main_result},
	we sketch our strategy for the proofs.
	In order to construct a local Ricci
	flow in Theorem \ref{main_result}, we do the conformal changing method inductively which was introduced in \cite{T3}, one may also see \cite{Lai} and \cite{LP} for the use of this method.
	 The process starts by doing a
	conformal change to the initial metric, making it a complete metric with bounded curvature
	and leaving it unchanged on a smaller region, and then run a complete Ricci flow up to a
	short time by using Shi’s classical existence theorem from \cite{Sh}. Next we do the conformal
	change to the metric again and repeating the process.
		This process led to
	define sequences of times $t_k$ and radii $r_k$ inductively:
	$t_{k+1}=\left(1+C_1\right) t_k,
	r_{k+1}=r_k-C_2 t_k^{\frac{1}{2}}$ with  uniform constants $C_1$ and $C_2$. In each step, by the Shi's short-time existence theorem \cite{Sh} and $|Rm(g(t_k))|\le \frac{\alpha}{t_k}$ by the inductive assumption we can get a prior estimate
	\begin{equation}\label{prior_estimates}
		|Rm(g(t))|\le \frac{Q}{t}
	\end{equation}
	on $[t_k,t_{k+1}]$ in a smaller region  for some possibly large constant $Q$.
	And the key step in our proof is to prove the
	local $\boldsymbol{\nu}$-functional keeps almost Euclidean in the above process which will imply
	\begin{equation}
		|Rm(g(t))|\le \frac{\alpha}{t}
	\end{equation}
	on $[t_k,t_{k+1}]$ in a smaller region (see Theorem \ref{key_lemma}) which lead the induction  to the next step.	
	Notice that
	the estimates for local $\nu$-functional values under the Ricci flows in Lemma \ref{key_lemma}) only hold on complete and smooth case; see Theorem \ref{Harnack}.    We should estimate the difference of the local $\nu$-functional values  on
	each step and prove the sum of these difference is  almost Euclidean.

	Now we can give the proofs of Theorem \ref{main_result} and Theorem \ref{main_result2}.
	
	\textbf{Proof of Theorem \ref{main_result}.}
	Firstly, an immediate consequence of Lemma \ref{Ho} and Shi's existence theorem for Ricci flows starting with complete initial metrics of bounded curvature \cite{Sh} is the following :
	If $\left(N^n, h_0\right)$ is a smooth manifold (not necessarily complete) that satisfies $\left|\operatorname{Rm}\left(h_0\right)\right| \leq \rho^{-2}$ throughout for some $\rho>0$, then there exist constants $\beta(n)$, $\Lambda(n)$ and a complete smooth Ricci flow $h(t)$ on $N$ for $t \in\left[0, \beta \rho^2\right]$ such that $h(0)=h_0$ on $N_\rho=\left\{x \in N: B_{h_0}(x, \rho) \Subset N\right\}$ and $|\operatorname{Rm}(h(t))| \leq \Lambda \rho^{-2}$ throughout $N \times\left[0, \beta \rho^2\right]$.

	Up to the rescaling, we can assume $T=1$ without loss of generality. Denote $10A=\delta^{-1}$ and
	take a constant $Q\ge \Lambda
	(\alpha+\beta)$.
	Choose $\eta \delta^{-1} > \rho_0>0$ sufficiently small so that  $\left|\operatorname{Rm}\left(g_0\right)\right| \leq \rho_0^{-2}$ on $ B_{g_0}(x_0, 10A)$. Applied with $N=B_{g_0}(x_0, 10A)$, we can find a complete smooth solution $h_1(t)$ to the Ricci flow on  $B_{g_0}(x_0, 10A)\times\left[0, \beta \rho_0^2\right]$ with $$|\operatorname{Rm}(h_1(t))| \leq \Lambda \rho_0^{-2} \text{\quad on\quad} B_{g_0}(x_0, 10A)\times\left[0, \beta \rho_0^2\right] $$
	and
	$$h_1(\cdot,0)=g_0 \text{\quad on\quad}  B_{g_0}(x_0, 10A-\rho_0).$$
	Then we denote $g(t)=h_1(t)$ on $B_{g_0}(x_0, 10A-\rho_0)\times\left[0, \beta \rho_0^2\right]$.
	Because $Q \geq \Lambda \beta$, the curvature bound can be weakened to
	\begin{equation}\label{prior_estimate}
		|\operatorname{Rm}(h_1(t))| \leq Q t^{-1} \text{\quad on \quad} B_{g_0}(x_0, 10A)\times\left[0, \beta \rho^2_0\right].
	\end{equation}

	Then we rescale the Ricci flow $h_1(t)$ as  $\tilde{h}_1(t)=t_1^{-1}h_1(t_1t)$,  $t\in [0,1]$,  where $t_1=\beta \rho_0^2$. Now we consider the ball $ B_{g_0}(x_0, r_1)$ with $r_1=10(1-t_1^{\frac{1}{2}})A-\rho_0$.
	Applying Lemma \ref{key_lemma} to the complete Ricci flow $\tilde{h}_1(t)$  with $s_1=0$,
	$s_2=t$, $B=\frac{1}{4}$, $D=0$ and $\tau_1=1$, we get for any $x\in B_{g_0}(x_0, r_1)$ and $t\in [0,1]$
	\begin{align*}
		&\boldsymbol{\nu}\left(B_{\tilde{h}_1(t)}(x, 3A), \tilde{h}_1(t), 2-t\right)-\boldsymbol{\nu}\left(B_{\tilde{h}_1(0)}(x, 10A), \tilde{h}_1(0), 2\right)\\
		\geq&\boldsymbol{\nu}\left(B_{\tilde{h}_1(t)}(x,5A-2A\sqrt{t}), \tilde{h}_1(t), 2-t\right)-\boldsymbol{\nu}\left(B_{\tilde{h}_1(0)}(x, 10A), \tilde{h}_1(0), 2\right)\\
		\geq&-\left\{\frac{1}{5 A^2B^2}+e^{-1}\right\} \cdot\left\{e^{\frac{t }{10 A^2B^2}}-1\right\}\\
		\geq&-\left\{\frac{1}{5 A^2B^2}+e^{-1}\right\} \cdot\left\{e^{\frac{1 }{10 A^2B^2}}-1\right\}\\
		\geq& -A^2,
	\end{align*}
	when $A$ is large. Without loss of generality, we can assume $t_1<\frac{1}{2}$.
	Then we have for any $x\in B_{g_0}(x_0, r_1)$ and $t\in [0,1]$
	\begin{align*}
		&\boldsymbol{\nu}\left(B_{\tilde{h}_1(t)}(x, 3A), \tilde{h}_1(t), 2-t\right)\\
		\geq& \boldsymbol{\nu}\left(B_{\tilde{h}_1(0)}(x, 10A), \tilde{h}_1(0), 2\right)-A^2\\
		=& \boldsymbol{\nu}\left(B_{g_0}(x, 10t_1^{\frac{1}{2}}A), g_0, 2t_1\right)-A^2\\
		\ge& \boldsymbol{\nu}\left(B_{g_0}(x_0, 10A), g_0, 2t_1\right)-A^2\\
		\ge  & \boldsymbol{\nu}\left(B_{g_0}(x_0, 10A, g_0, 1\right)-A^2\\
		\ge  & -101A^{-2},
	\end{align*}
	where we use $h_1(\cdot,0)=g_0$ on $B_{g_0}(x_0, 10A-\rho_0)$ and  $B_{g_0}(x, 10t_1^{\frac{1}{2}}A) \subset B_{g_0}(x_0, 10A-\rho_0)$ in the above inequalities.
	
	Next we prove that there exists a positive constant $A_{\alpha}$ depending only on $\alpha$ and $n$ such that
	$$|\operatorname{Rm}_{\tilde{h}_1}(x,t)| \leq \frac{\alpha}{t}$$
	for any  $x\in B_{g_0}(x_0, r_1)$ and $t\in (0,1]$  if $A
	\ge A_{\alpha}$. Otherwise, there exist a sequence of Ricci flows $h^i_1(t)|_{t\in [0,1]} $
	such that $\boldsymbol{\nu}\left(B_{\tilde{h}^i_1(t)}(x_i, 3A_i), \tilde{h}^i_1(t), 2-t\right)\ge -101A_i^{-2}$ with $A_i\to \infty$. Moreover, up to rescaling, we may assume
	$|\operatorname{Rm}_{\tilde{h}^i_1}(x_i,1)|=\alpha_0$ for some $\alpha_0>0$. Since
	$$|\operatorname{Rm}_{\tilde{h}^i_1}| \leq \frac{Q}{t}$$
	on $B_{\tilde{h}^i_1(t)}(x_i, 3A_i)\times\left(0, 1\right]$ by (\ref{prior_estimate}) and the non-collapsing by Theorem 3.3 in \cite{w1},
	we have $\left(B_{\tilde{h}^i_1(t)}(x_i, 3A_i), \tilde{h}^i_1(t), x_i\right)$ subconverges to a complete Ricci flow $\left(M^{\infty}, \tilde{h}^{\infty}_1(t), x_{\infty}\right)$ in $C^{\infty}$ sense with $|\operatorname{Rm}_{\tilde{h}^{\infty}_1}(x_{\infty},1)|=\alpha_0$
	with $\boldsymbol{\nu}\left(M^{\infty}, \tilde{h}^{\infty}_1(t),2-t\right)\ge 0$. Then  $\left(M^{\infty}, \tilde{h}^{\infty}_1(t)\right)$ must be isometric to the Euclidean space by Proposition 3.2 in \cite{w1}, which is a contradiction.

	We now define the sequences of times $t_k$ and radii $r_k$ inductively as follows:
	
	(a) $t_0=0$, $t_1=\beta \rho_0^2$ and $t_{k+1}=\left(1+\beta \alpha^{-1}\right) t_k$ for $k \geq 1$;
	
	(b) $r_0=10A-\rho_0$, $r_1=10A-\rho_0-10t_1^{\frac{1}{2}}A$, and $r_{k}=10A-\rho_0-10A\sum\limits^{k}_{i=1}t_i^{\frac{1}{2}}-(\alpha^{-\frac{1}{2}}+2\gamma\alpha^{\frac{1}{2}})\sum\limits^{k-1}_{i=1}t_i^{\frac{1}{2}}$ for $k \geq 2$.
	
	Let $\mathcal{P}(k)$ be the following statement: there exist a complete smooth Ricci flow $h_k(t)$ on time interval $\left[t_{k-1}, t_k\right]$ with 	
	\begin{equation*}
		|\operatorname{Rm}(h_{k+1}(t))|  \leq \frac{Q}{t}
	\end{equation*}
and a Ricci flow $g(t)$ on time interval $\left[0, t_k\right]$
	with $$g(t_{k-1})=h_k(t_{k-1})\text{\quad on \quad}B_{g_0}\left(x_0,r_{k-1}-(\alpha^{-\frac{1}{2}}+\gamma\alpha^{\frac{1}{2}})t_{k-1}^{\frac{1}{2}}\right)$$
	and
	$$
	g(t)=h_k(t)\text{\quad on \quad} B_{g_0}\left(x_0,r_{k-1}-(\alpha^{-\frac{1}{2}}+\gamma\alpha^{\frac{1}{2}})t_{k-1}^{\frac{1}{2}}\right)\times \left[t_{k-1}, t_k\right].
	$$
	Moreover,  $g(t)$ is smooth on
	$B_{g_0}\left(p, r_k\right) \times\left[0, t_k\right]$ and  satisfies
	$$|\operatorname{Rm}(g(t))| \leq \frac{\alpha}{t}\text{\quad on \quad} B_{g_0}\left(p, r_k\right) \times\left[0, t_k\right] $$
	with $g(0)=g_0$ on $B_{g_0}\left(p, r_k\right)$.
	Noted that we have proved that $\mathcal{P}(1)$ is true. Our goal is to show that $\mathcal{P}(k)$ is true for all $k$ provided $r_k>0$.
	
	We now perform an inductive argument. Suppose $\mathcal{P}(k)$ is true, we have a smooth Ricci flow $ g(t)$ on $B_{g_0}\left(p, r_k\right) \times\left[0, t_k\right]$ with $|\operatorname{Rm}(g(t))| \leq\frac{\alpha}{t}$. Applying Theorem \ref{Ho} with
	$N=B_{g_0}\left(x_0, r_k\right)$ so that for $h=g\left(t_k\right)$, we have
	$$\sup\limits _{N}\left|\operatorname{Rm}\left(h\right)\right| \leq \rho^{-2},$$
	where $\rho=\sqrt{t_k \alpha^{-1}}$. Moreover, for any $x \in B_{g_0}\left(x_0, r_k-(\alpha^{-\frac{1}{2}}+\gamma \alpha^{\frac{1}{2}}) t_k^{\frac{1}{2}}\right)$, Lemma \ref{shrinking}  gives
	$$\quad B_{g\left(t_k\right)}(x, \rho)  \subset B_{g_0}\left(x, (\alpha^{-\frac{1}{2}}+\gamma \alpha^{\frac{1}{2}}) t_k^{\frac{1}{2}}\right) \Subset N.$$
	This shows that $B_{g_0}\left(x_0, r_k-(\alpha^{-\frac{1}{2}}+\gamma \alpha^{\frac{1}{2}}) t_k^{\frac{1}{2}}\right) \subset N_\rho=\{x\in N| B_{g\left(t_k\right)}(x, \rho)\subset N\}$. Hence, we can find a complete Ricci flow $h_{k+1}(t)$ on  $B_{g_0}\left(p, r_{k}\right) \times\left[t_k, t_k+\beta \rho^2\right]$ with
	\begin{equation}\label{prior_cur_est}
		|\operatorname{Rm}(h_{k+1}(t))| \leq \Lambda \rho^{-2}=\Lambda \alpha t_k^{-1} \leq Q t^{-1}
	\end{equation}
	since $\Lambda\left(\alpha+\beta\right) \leq Q$, and
	$$h_{k+1}(t_{k})=g(t_{k})\text{\quad on \quad} B_{g_0}\left(x_0,r_{k}-(\alpha^{-\frac{1}{2}}+\gamma\alpha^{\frac{1}{2}})t_{k}^{\frac{1}{2}}\right)$$ and
	$t_k+\beta \rho^2=t_k\left(1+\beta \alpha^{-1}\right)=t_{k+1}$.
	Then we denote
	\begin{equation}\label{def_g}
	g(t)=h_{k+1}(t)\text{\quad on \quad} B_{g_0}\left(x_0,r_{k}-(\alpha^{-\frac{1}{2}}+\gamma\alpha^{\frac{1}{2}})t_{k}^{\frac{1}{2}}\right)\times \left[t_{k}, t_{k+1}\right].
	\end{equation}
	
	For $x\in B_{g_0}\left(x_0, r_{k+1}\right)$, together with Lemma \ref{shrinking} give for $i<k+1$
	$$\quad B_{g\left(t_i\right)}(x, 10t_{k+1}^{\frac{1}{2}}A)  \subset B_{g_0}\left(x, 10t_{k+1}^{\frac{1}{2}}A+\gamma\alpha^{\frac{1}{2}} t_{i}^{\frac{1}{2}}\right) \subset B_{g_0}\left(x_0,r_{k+1}+ 10t_{k+1}^{\frac{1}{2}}A+\gamma\alpha^{\frac{1}{2}} t_{i}^{\frac{1}{2}}\right)\subset B_{g_0}\left(x_0, r_i-(\alpha^{-\frac{1}{2}}+\gamma\alpha^{\frac{1}{2}}) t_i^{\frac{1}{2}}\right),$$
	by the definition of $r_{k+1}$. Then we have
	\begin{equation}\label{equal}
		g(t_i)=h_{i+1}(t_i)=h_{i}(t_i)
	\end{equation}
	on   $B_{g\left(t_i\right)}(x, 10t_{k+1}^{\frac{1}{2}}A)$ for any $i<k+1$.
	
	We rescale $g(t)$ and $h_i(t)$ as $\tilde{g}(t)=t_{k+1}^{-1}g(t_{k+1}t)|_{t\in [0,1]}$, $\tilde{h}_{i}(t)=t_{k+1}^{-1}h_{i}(t_{k+1}t)|_{t\in [\tilde{t}_{i-1},\tilde{t}_i]}$ for any
	$i\le k+1$, where
	$\tilde{t}_i=t_{k+1}^{-1}t_i=(1+\beta\alpha^{-1})^{i-k-1}$ for $i\ge 1$ and $\tilde{t}_0=0$.
	Denote
	$R_0=10$, $R_1=10-2\tilde{t}_1^{\frac{1}{2}}$ and $R_i=10-2\tilde{t}_i^{\frac{1}{2}}-\frac{5}{M_{\alpha}}(\beta\alpha^{-1})^\frac{1}{5}\sum\limits^{i-1}_{j=0}\tilde{t}_j^{\frac{1}{5}}$ for $i\ge 2$, where $M_{\alpha}=\frac{(\beta\alpha^{-1})^\frac{1}{5}}{(1+\beta\alpha^{-1})^\frac{1}{5}-1}$.
	Applying Theorem \ref{key_lemma} to the complete Ricci flows $\tilde{h}_{i+1}(t)|_{t\in [\tilde{t}_{i},\tilde{t}_{i+1}]}$ with  $s_1=\tilde{t}_{i}$, $s_2=\tilde{t}_{i+1}$,  $B=\frac{1}{4M_{\alpha}}(\tilde{t}_{i+1}-\tilde{t}_{i})^{\frac{1}{5}}=\frac{1}{4M_{\alpha}}(\beta\alpha^{-1})^\frac{1}{5}\tilde{t}_{i}^{\frac{1}{5}}$, $\tau_1=1$, $D=\frac{5}{M_{\alpha}}(\beta\alpha^{-1})^\frac{1}{5}\sum\limits^{i-1}_{j=0}\tilde{t}_j^{\frac{1}{5}}$ when $i\ge 1$ and $D=0$ when $i=0$,
	we have for any $x\in B_{g_0}\left(x_0, r_{k+1}\right)$ and $0\le i\le k-1$
	\begin{equation}\label{key}
		\begin{aligned}
	&\boldsymbol{\nu}\left(B_{\tilde{g}(\tilde{t}_{i+1})}(x, R_{i+1} A ),\tilde{g}(\tilde{t}_{i+1}),2-\tilde{t}_{i+1}\right)-	\boldsymbol{\nu}\left(B_{\tilde{g}(\tilde{t}_{i})}(x, R_{i} A), \tilde{g}(\tilde{t}_{i}),2-\tilde{t}_{i}\right)\\
	=	&\boldsymbol{\nu}\left(B_{\tilde{h}_{i+1}(\tilde{t}_{i+1})}(x, R_{i+1} A ),\tilde{h}_{i+1}(\tilde{t}_{i+1}),2-\tilde{t}_{i+1}\right)-	\boldsymbol{\nu}\left(B_{\tilde{h}_{i+1}(\tilde{t}_{i})}(x, R_{i} A), \tilde{h}_{i+1}(\tilde{t}_{i}),2-\tilde{t}_{i}\right)\\
	\geq &-\left(\frac{16M^2_{\alpha}}{5A^2(\tilde{t}_{i+1}-\tilde{t}_{i})^{\frac{2}{5}}}+e^{-1}\right)\left(e^{\frac{8M^2_{\alpha}(\tilde{t}_{i+1}-\tilde{t}_{i})^{\frac{3}{5}}}{5A^2}}-1\right)
	\\	
	\geq &-\left(\frac{16M^2_{\alpha}}{5A^2(\tilde{t}_{i+1}-\tilde{t}_{i})^{\frac{2}{5}}}+e^{-1}\right)\frac{8eM^2_{\alpha}(\tilde{t}_{i+1}-\tilde{t}_{i})^{\frac{3}{5}}}{5A^2}
	\\	
	= &-\left(\frac{128eM^2_{\alpha}}{25A^2}+\frac{8}{5}(\tilde{t}_{i+1}-\tilde{t}_{i})^{\frac{2}{5}}\right)\frac{M^2_{\alpha}(\tilde{t}_{i+1}-\tilde{t}_{i})^{\frac{1}{5}}}{A^2}
	\\	
	\geq &-2M_{\alpha}^2(\tilde{t}_{i+1}-\tilde{t}_{i})^{\frac{1}{5}}A^{-2}=-2M_{\alpha}^2(\beta\alpha^{-1})^\frac{1}{5}\tilde{t}_{i}^{\frac{1}{5}}A^{-2},
		\end{aligned}	
	\end{equation}
	when $A\ge \frac{64eM^2_{\alpha}}{5}$, where we use (\ref{equal}) and $\tilde{t}_{i}\le 1$ for
	$i\le k+1$ in the above inequalities.
	Likewise, applying Theorem \ref{key_lemma} to the complete Ricci flows $\tilde{h}_{k+1}(t)|_{t\in [\tilde{t}_{k},\tilde{t}_{k+1}]}$ with  $s_1=\tilde{t}_{k}$, $s_2=t$,  $B=\frac{1}{4M_{\alpha}}(\tilde{t}_{k+1}-\tilde{t}_{k})^{\frac{1}{5}}=\frac{1}{4M_{\alpha}}(\beta\alpha^{-1})^\frac{1}{5}\tilde{t}_{k}^{\frac{1}{5}}$, $\tau_1=1$, $D=\frac{5}{M_{\alpha}}(\beta\alpha^{-1})^\frac{1}{5}\sum\limits^{k-1}_{j=0}\tilde{t}_j^{\frac{1}{5}}$, we have
	for $x\in B_{g_0}\left(x_0, r_{k+1}\right)$ and $t\in [\tilde{t}_k,\tilde{t}_{k+1}]$,
	\begin{align*}
		&\boldsymbol{\nu}\left(B_{\tilde{h}_{k+1}(t)}(x, R_{k+1} A ),\tilde{h}_{k+1}(t),2-t\right)-	\boldsymbol{\nu}\left(B_{\tilde{h}_{k+1}(\tilde{t}_{k})}(x, R_{k} A), \tilde{h}_{k+1}(\tilde{t}_{k}),2-\tilde{t}_{k}\right)\\
		\ge &\boldsymbol{\nu}\left(B_{\tilde{h}_{k+1}(t)}(x, (10-2t^{\frac{1}{2}}-\frac{5}{M_{\alpha}}(\beta\alpha^{-1})^\frac{1}{5}\sum\limits^{k}_{j=0}\tilde{t}_j^{\frac{1}{5}}) A ),\tilde{h}_{k+1}(t),2-t\right)-	\boldsymbol{\nu}\left(B_{\tilde{g}(\tilde{t}_{k})}(x, R_{k} A), \tilde{g}(\tilde{t}_{k}),2-\tilde{t}_{k}\right)\\
		\geq&-\left\{\frac{1}{5 A^2B^2}+e^{-1}\right\} \cdot\left\{e^{\frac{t-\tilde{t}_{k} }{10 A^2B^2}}-1\right\}\\
		\geq&-\left\{\frac{1}{5 A^2B^2}+e^{-1}\right\} \cdot\left\{e^{\frac{\tilde{t}_{k+1}-\tilde{t}_{k} }{10 A^2B^2}}-1\right\}\\	
		\geq &-2M_{\alpha}^2(\beta\alpha^{-1})^\frac{1}{5}\tilde{t}_{k}^{\frac{1}{5}}A^{-2},
	\end{align*}
where we use the same estimates as (\ref{key}) to get the last inequality.
	Notice that $R_{k+1}=10-2-\frac{5}{M_{\alpha}}(\beta\alpha^{-1})^\frac{1}{5}\sum\limits^{k}_{j=1}\tilde{t}_j^{\frac{1}{5}}\ge 3$ and we can assume $t_{k+1}<\frac{1}{2}$ without loss of generality.
	It follows that
	for any $x\in B_{g_0}\left(x_0, r_{k+1}\right)$ and $t\in [\tilde{t}_k,\tilde{t}_{k+1}]=[(1+\beta\alpha^{-1})^{-1},1]$, we have

	\begin{align}
		&\boldsymbol{\nu}\left(B_{\tilde{h}_{k+1}(t)}(x,3 A ),\tilde{h}_{k+1}(\tilde{t}),2-t\right)\nonumber\\
		\ge&\boldsymbol{\nu}\left(B_{\tilde{h}_{k+1}(t)}(x, R_{k+1} A ),\tilde{h}_{k+1}(\tilde{t}),2-t\right)\nonumber\\
		\ge& \boldsymbol{\nu}\left(B_{\tilde{g}(0)}(x, 10A), \tilde{g}(0), 2\right)-2M_{\alpha}^2(\beta\alpha^{-1})^\frac{1}{5}\sum\limits_{i=1}^{k}\tilde{t}_{i}^{\frac{1}{5}}A^{-2}\nonumber\\
		=& \boldsymbol{\nu}\left(B_{g_0}(x, 10t_{k+1}^{\frac{1}{2}}A), g_0, 2t_{k+1}\right)-2M_{\alpha}^2(\beta\alpha^{-1})^\frac{1}{5}\sum\limits_{i=1}^{k}\tilde{t}_{i}^{\frac{1}{5}}A^{-2}\nonumber\\
		\ge & \boldsymbol{\nu}\left(B_{g_0}(x_0, 10A), g_0, 1\right)-2M_{\alpha}^3A^{-2}\nonumber\\
		\ge& -(2M_{\alpha}^3+100)A^{-2}\label{mu_estimate},
	\end{align}
	where we use $(\beta\alpha^{-1})^\frac{1}{5}\sum\limits_{i=1}^{k}\tilde{t}_{i}^{\frac{1}{5}}<M_{\alpha}$ and $B_{g_0}(x, 10t_{k+1}^{\frac{1}{2}}A)\subset B_{g_0}(x_0, 10A) $ in the above inequalities. Combining with (\ref{prior_cur_est}) and (\ref{mu_estimate}), we can use the same contradiction arguments as $k=1$ to prove that there exists a positive constant $A_{\alpha}$ depending on $\alpha$ and $n$ such that $$|\operatorname{Rm}_{\tilde{h}_{k+1}}(x,t)| \leq \frac{\alpha}{t}$$
	for $x\in B_{g_0}\left(x_0, r_{k+1}\right)$ and $t\in [\tilde{t}_k,\tilde{t}_{k+1}]$  if $A
	\ge A_{\alpha}$. This shows $|\operatorname{Rm}(g(t))| \leq\frac{\alpha}{t}$ on $B_{g_0}\left(p, r_{k+1}\right) \times\left[t_k, t_{k+1}\right]$ by (\ref{def_g}).
	Hence $\mathcal{P}(k+1)$ is true provided that $r_{k+1}>0$.
	
	Since $\lim\limits_{j \rightarrow+\infty} r_j=-\infty$, for any $\eta\in (0,1)$, there is $k \in \mathbb{N}$ such that $r_k \geq 10(1-\eta)A$ and $r_{k+1}<10(1-\eta)A$. In particular, $\mathcal{P}(k)$ is true since $r_k>0$. We now estimate $t_k$:
	\begin{align*}
		10(1-\eta)A> &r_{k+1}=10A-10A\sum\limits^{k+1}_{i=1}t_i^{\frac{1}{2}}-2(\alpha^{-\frac{1}{2}}+2\gamma\alpha^{\frac{1}{2}})\sum\limits^{k}_{i=1}t_i^{\frac{1}{2}} -\rho_0\\
		&\ge 10A-12A\sum\limits^{k+1}_{i=1}t_i^{\frac{1}{2}} \\
		& \geq 10A-12A t_{k+1}^{\frac{1}{2}}\sum_{i=1}^{\infty}(1+\beta\alpha^{-1})^{-\frac{i}{2}} \\
		&=10A-  \frac{12At_{k+1}^{\frac{1}{2}}}{(1+\beta\alpha^{-1})^{\frac{1}{2}}-1},
	\end{align*}
	when $A> \alpha^{-\frac{1}{2}}+2\gamma\alpha^{\frac{1}{2}}$ and $\rho_0$ is sufficient small.
	This implies
	$$
	t_{k+1}>\frac{ 25\eta^2}{36((1+\beta\alpha^{-1})^{\frac{1}{2}}-1)^2}=: \epsilon(\alpha,n)^2\eta^2 .
	$$
	In other words,  for any $\eta\in (0,1)$ there exists a smooth Ricci flow solution $g(t)$ defined on $B_{g_0}(x_0, 10(1-\eta)A) \times[0, \epsilon(\alpha,n)^2\eta^2]$ so that $g(0)=g_0$ and $|\operatorname{Rm}(g(t))| \leq \frac{\alpha}{t}$ if $A\ge A_{\alpha}$. And
	(\ref{non_collapsing}) follows from the estimate (\ref{mu_estimate}) and Theorem 3.3 in \cite{w1}.
	This completes the proof.
	$\Box$

	\textbf{Proof of Theorem \ref{main_result2}.}
	Up to rescaling, we may assume $r_0=1$ without loss of generality.
	Now we let $T=\delta^2$. For any $\tilde{\Omega}\subset \mathbb{R}^n$, we have $\boldsymbol{\mu}\left(\tilde{\Omega}, g_E, \tau \right)\ge \boldsymbol{\mu}\left(\mathbb{R}^n, g_E, \tau\right)\ge 0$.
	By Theorem \ref{isoperimetric}, we get for any $t<T=\delta^2$
	\begin{align*}
		&\boldsymbol{\mu}\left(B_{g_0}(x_0,\delta^{-1}\sqrt{T} ),g_0,t\right)\\
		=&\boldsymbol{\mu}\left(B_{g_0}(x_0,1 ),g_0,t\right)\\
		\ge& n\log(1-\delta)-\delta^2\\
		\ge& -2n\delta-\delta^2,
	\end{align*}
	when $\delta<\frac{1}{2}$. It follows that $\boldsymbol{\nu}\left(B_{g_0}(x_0,\delta^{-1}\sqrt{T} ),g_0,t\right)\ge -2n\delta-\delta^2$.
	Then Theorem \ref{main_result2} follows by Theorem \ref{main_result} directly.
	
	$\Box$

	\section{The applications to the incomplete pseudolocality theorems}
	The proof of Theorem \ref{existence} relies on the following pseudolocality theorems for incomplete case.
	\begin{thm}\label{local}
		For each $\alpha>0$ and $n\ge 2$, there exist $\delta=\delta
		(\alpha,n)$ and $\epsilon(\alpha,n)$ with the following properties.
		Suppose $\left(M, g_0\right)$ is a smooth $n$-dimensional Riemannian manifold (not necessarily complete) such that $B_{g_0}\left(x_0, K\delta^{-1} \sqrt{T}\right)\Subset M$ for $K>1$ and $T>0$. Moreover,
		for any
		$x\in B_{g_0}\left(x_0, (K-1)\delta^{-1} \sqrt{T}\right)$
		we have
		\begin{equation}\label{local_mu}
			\nu\left(B_{g_0}\left(x,\delta^{-1} \sqrt{T}\right), g_0, T\right) \geq-\delta^2.
		\end{equation}
		Then for each $\eta \in(0, 1)$ there exists a smooth Ricci flow $g(t)$ on $B_{g_0}\left(x_0, (K-1) (1-\eta)\delta^{-1} \sqrt{T}\right)\times [0, (\epsilon\eta)^2T]$ with $g(0)=g_0$ satisfying
		$$|R m|(x, t) \leq \frac{\alpha}{t}$$
		and
		$$\inf _{\rho \in\left(0, \alpha^{-1} \sqrt{t}\right)} \frac{\operatorname{Vol}\left(B_{g(t)}(x, \rho)\right)}{\rho^{n}} \geq(1-\alpha) \omega_n$$
		for  $(x,t)\in B_{g_0}\left(x_0, (K-1) (1-\eta)\delta^{-1} \sqrt{T}\right)\times [0, (\epsilon\eta)^2T]$.
	\end{thm}
	\begin{proof}
		We can assume $T=1$ without loss of generality. Denote $\delta^{-1}=10A$.
		We only need modify the definitions of the sequence $r_k$ in the proof of Theorem \ref{main_result} to the following: $r_0=10(K-1)A$, $r_1=10(K-1-t_1^{\frac{1}{2}})A$, and $r_{k}=10(K-1)A-10A\sum\limits^{k}_{i=1}t_i^{\frac{1}{2}}-(\alpha^{-\frac{1}{2}}+2\gamma\alpha^{\frac{1}{2}})\sum\limits^{k-1}_{i=1}t_i^{\frac{1}{2}}$ for $k \geq 2$. Also noted that if $x\in B_{g_0}(x,r_{k+1})\subset B_{g_0}(x,10(K-1)A)$ and $t_{k+1}<\frac{1}{2}$, we have $\boldsymbol{\nu}\left(B_{g_0}(x, 10t_{k+1}^{\frac{1}{2}}A), g_0, 2t_{k+1}\right)\ge \boldsymbol{\nu}\left(B_{g_0}(x, 10A), g_0, 1\right)\ge -100A^2$ by (\ref{local_mu}). Then  the estimates in (\ref{mu_estimate}) still go through in this case.
		Since the rest of proof is almost same as Theorem \ref{main_result},
		we leave the details to
		the readers.
	\end{proof}

	\begin{cor}\label{cor}
		For every $\alpha>0$, $n\ge 2$ and $r_0>0$,  there exist $\delta=\delta
		(\alpha,n)$ and $\epsilon(\alpha,n)$ with the following properties.
		Suppose $\left(M, g_0\right)$ is a smooth $n$-dimensional Riemannian manifold(not necessarily complete) such that $B_{g_0}\left(x_0, Kr_0\right)\Subset M$ for some $K>0$. Moreover,  for any $x\in B_{g_0}\left(x_0, (K-1)r_0\right)$
		we have
		$$R\geq-r_0^{-2}  \text{\quad on\quad}  B_{g_0}\left(x, r_0\right)$$
		and
		$$
		\left(\text { Area }_{g_0}(\partial \Omega)\right)^n \geq(1-\delta) c_n\left(\operatorname{Vol}_{g_0}(\Omega)\right)^{n-1}
		$$
		for any regular domain $\Omega \subset B_{g_0}\left(x, r_0\right)$.
		Then for each $\eta \in(0, 1)$ there exists a smooth Ricci flow $g(t)$ on $B_{g_0}\left(x_0,(K-1)(1-\eta) r_0\right) \times\left[0,\left(\epsilon \eta r_0\right)^2\right]$ with $g(0)=g_0$ satisfying
		$$
		|\mathrm{Rm}|(x,t) \leq \frac{\alpha}{t}+\frac{1}{\left(\epsilon \eta r_0\right)^2}
		$$
		and
		$$\inf _{\rho \in\left(0, \alpha^{-1} \sqrt{t}\right)} \frac{\operatorname{Vol}\left(B_{g(t)}(x, \rho)\right)}{\rho^{n}} \geq(1-\alpha) \omega_n$$
		for
		$(x,t)\in B_{g_0}\left(x_0,(K-1)(1-\eta)\right) \times\left(0,\left(\epsilon \eta r_0\right)^2\right]$.
	\end{cor}
	\begin{proof}
		Corollary \ref{cor} follows from Theorem \ref{local} and Theorem \ref{isoperimetric} just as the proof of Theorem \ref{main_result2}.
	\end{proof}

	Now we give the proof of Theorem \ref{existence}. Indeed, we prove
	a stronger version. 	And Theorem \ref{existence} is just a direct corollary of Theorem \ref{local_existence} and Theorem \ref{isoperimetric}.
	
	 \begin{thm}\label{local_existence}
		For each $\alpha>0$ and $n\ge 2$, there exist $\delta=\delta
		(\alpha,n)$ and $\epsilon(\alpha,n)$ with the following properties.
		Suppose $\left(M^n, g_0\right)$ is a smooth complete $n$-dimensional Riemannian manifold  such that
		\begin{equation}\label{local_mu}
			\nu\left(B_{g_0}\left(x,\delta^{-1} \sqrt{T}\right), g_0, T\right) \geq-\delta^2.
		\end{equation}
		for any
		$x\in M$
		and some $T>0$.
		Then for each $\eta \in(0, 1)$ there exists a smooth Ricci flow $g(t)$ on $M\times [0, (\epsilon\eta)^2T]$ with $g(0)=g_0$ satisfying
\begin{equation}\label{local_1}
|R m|(x, t) \leq \frac{\alpha}{t}
\end{equation}		
		and
\begin{equation}\label{local_2}
\inf _{\rho \in\left(0, \alpha^{-1} \sqrt{t}\right)} \frac{\operatorname{Vol}\left(B_{g(t)}(x, \rho)\right)}{\rho^{n}} \geq(1-\alpha) \omega_n
\end{equation}
	for  $(x,t)\in M\times [0, (\epsilon\eta)^2T]$.
	\end{thm}

\begin{rem}
	Wang \cite{w2}  proved under a stronger assumption that if $\left(M^n, g_0\right)$ is a smooth complete $n$-dimensional Riemannian manifold  such that	 $\min \left\{\nu(M, g, T), nT R c_{\text {min }}(x)\right\} \geq-\delta^2$ then there exists a smooth Ricci flow $g(t)$ on $M\times [0, T]$ with $g(0)=g_0$ satisfying (\ref{local_1}), (\ref{local_2}) and the following distortion estimates hold:
\begin{align*}
\begin{aligned}
& \left|\log \frac{d_{g(t)}(x, y)}{d_{g(0)}(x, y)}\right|<\psi\left\{1+\log _{+} \frac{\sqrt{t}}{d_{g(0)}(x, y)}\right\}, \quad \forall t \in(0, T), x, y \in M . \\
& \left|d_{g(0)}(x, y)-d_{g(t)}(x, y)\right|<\psi \sqrt{t}, \quad \forall t \in(0, T), x, y \in M, d_{g(0)}(x, y) \leq \sqrt{t} ; 
\end{aligned}
\end{align*}
see Corollary 5.5 in \cite{w2}.
\end{rem}

\begin{proof}
	Applying Theorem \ref{local} to $B_{g_0}(x_0,K_i\delta^{-1} \sqrt{T})$ for any $\eta>0$ and let $K_i\to \infty$, we get a sequence of Ricci flows $g_i(t)$ with $g_i(0)=g_0$ on $B_{g_0}(x_0,(K_i-1)(1-\eta)\delta^{-1} \sqrt{T})\times\left[0,(\epsilon\eta)^2T\right]$  satisfying
	$$
	|\operatorname{Rm}|(g_i(t)) \leq \frac{\alpha}{t}.
	$$
	Together with Shi's estimates \cite{Sh} and modified Shi's interior estimates \cite{LP}, $g_i$ subconverges to a
	smooth Ricci flow $g(t)$ on $M\times\left[0,(\epsilon\eta)^2T\right]$ with $g(0)=g_0$ satisfying
	$$
	|\operatorname{Rm}|(g(t)) \leq \frac{\alpha}{t},
	$$
	on $M\times\left[0,(\epsilon\eta)^2T\right]$.
	The completeness of $g(t)$ follows from Theorem
	\ref{shrinking}.
\end{proof}
	
	Finally, we give the proof of Theorem \ref{rigidity}. Indeed, we prove
	a stronger version which improves
	a result by Wang \cite{w2} with an extra condition  that $\left(M^n, g\right)$ has the bounded curvature; see Proposition 3.2 in \cite{w2}.
	And Theorem \ref{rigidity} is just a direct corollary of Theorem \ref{rigidity1} and Theorem \ref{isoperimetric}.
		\begin{thm}\label{rigidity1}
		Suppose $\left(M^n, g\right)$ is a smooth complete $n$-dimensional Riemannian manifold  such that
		$$\boldsymbol{v}(M, g, T) \geq 0$$
		for some $T>0$.
		Then $M$ is isometric to the Euclidean space.
	\end{thm}
\begin{proof}For any $\alpha>0$, applied with Theorem \ref{main_result} to $B_{g_0}\left(x,\delta_i^{-1} \sqrt{T}\right)$ with $\delta_i\to 0$ provides a sequence of Ricci flows $g_i(t)$ with $g_i(0)=g$ on
$B_{g_0}\left(x_0,  (1-\eta)\delta_i^{-1} \sqrt{T}\right)\times [0, (\epsilon\eta)^2T]$	for some $\eta>0$
	  satisfying
	$$
	|\operatorname{Rm}|(g_i(t)) \leq \frac{\alpha}{t}.
	$$
	Taking $i\to \infty$, together with Shi's estimates \cite{Sh} and modified Shi's interior estimates \cite{LP},  we get a complete smooth Ricci flow $g(t)$ on $ [0, (\epsilon\eta)^2T]$ and satisfying
	\begin{equation}\label{bbq}
	|\operatorname{Rm}|(g(t)) \leq \frac{\alpha}{t}.
	\end{equation}
	And we see from (\ref{mu_estimate}) that $\boldsymbol{\nu}(M,g(t),2-t)\ge 0$ for $t\le (\epsilon\eta)^2T$. And the curvature is bounded
	on $[t_0,(\epsilon\eta)^2T]$ for any $0<t_0<(\epsilon\eta)^2T$,
	then $(M,g(t))$ must be   isometric to the Euclidean space by Proposition 3.2 in \cite{w1} 	on $[t_0,(\epsilon\eta)^2T]$ for any $0<t_0<(\epsilon\eta)^2T$. It follows that  $(M,g(0))$ must be   isometric to the Euclidean space since $g(t)$ is smooth at $t=0$.
\end{proof}

\end{document}